\documentclass[12pt,a4paper,english]{amsart}
\usepackage{latexsym,mathtools}
\usepackage{epsfig,setspace}
\usepackage{a4}
\usepackage{amssymb}
\usepackage{amsthm}
\usepackage{amsbsy} 
\usepackage{upgreek}
\usepackage{relsize}
\usepackage{amsfonts,amssymb,latexsym,amscd,amsmath,euscript,enumerate,verbatim,calc,enumitem}
\usepackage[utf8]{inputenc}
\usepackage[cyr]{aeguill}
\usepackage[english]{babel}
\allowdisplaybreaks[1]
\usepackage{url} 
\usepackage{hhline}
\usepackage{pifont}
\usepackage[colorlinks=true,linkcolor=blue,citecolor=magenta]{hyperref}

\usepackage[style=numeric,firstinits=true,hyperref,useprefix,backend=bibtex]{biblatex}
\bibliography{references}
\usepackage{csquotes}

 \theoremstyle{plain}
\newtheorem*{thm}{Theorem}

\newtheorem{theorem}{Theorem}[section]
\newtheorem{lemma}[theorem]{Lemma}
\newtheorem{proposition}[theorem]{Proposition}

\theoremstyle{definition}

\newtheorem{remark}[theorem]{Remark}

\def\Gal{\mathop{\rm Gal}\nolimits}

\def\NS{\mathop{\rm NS}\nolimits}
\def\ev{\mathop{\rm ev}\nolimits}

\def\Pic{\mathop{\rm Pic}\nolimits}

\makeatletter
\def\imod#1{\allowbreak\mkern10mu({\operator@font mod}\,\,#1)}
\makeatother

\date{\today}

\begin{document}

\title[Algebraic geometry codes over surfaces]{Bounds on the minimum distance of algebraic geometry codes defined over some families of surfaces} 

\author{Yves Aubry, Elena Berardini, Fabien Herbaut and Marc Perret}

\thanks{Funded by ANR grant  ANR-15-CE39-0013-01 ``Manta"}

\newcommand{\Addresses}{{
  \bigskip
  \footnotesize

  Yves~Aubry, \textsc{Institut de Math\'ematiques de Toulon - IMATH,}\par\nopagebreak
  \textsc{Universit\'e de Toulon and Institut de Math\'ematiques de Marseille - I2M,}\par\nopagebreak
  \textsc{Aix Marseille Universit\'e, CNRS, Centrale Marseille, UMR 7373, France}\par\nopagebreak
  \textit{E-mail address}: \texttt{yves.aubry@univ-tln.fr}

  \medskip

  Elena~Berardini, \textsc{Institut de Math\'ematiques de Marseille - I2M,}\par\nopagebreak
  \textsc{Aix Marseille Universit\'e, CNRS, Centrale Marseille, UMR 7373, France}\par\nopagebreak
  \textit{E-mail address}: \texttt{elena\_berardini@hotmail.it}

  \medskip

  Fabien~Herbaut,  \textsc{INSPE Nice-Toulon,  Universit\'e C\^ote d'Azur,}\par\nopagebreak
  \textsc{Institut de Math\'ematiques de Toulon - IMATH, Universit\'e de Toulon, France}\par\nopagebreak
  \textit{E-mail address}: \texttt{fabien.herbaut@univ-cotedazur.fr}
  
   \medskip
   
  Marc~Perret, \textsc{Institut de Math\'ematiques de Toulouse, UMR 5219,}\par\nopagebreak
  \textsc{Universit\' e de Toulouse, CNRS, UT2J, F-31058 Toulouse, France}\par\nopagebreak
  \textit{E-mail address}: \texttt{perret@univ-tlse2.fr}

}}

\subjclass[2000]{14J99, 14G15, 14G50}

\keywords{AG codes, algebraic surfaces, fibered surfaces, finite fields.}

\maketitle
{\centering \textit{To the memory of Gilles Lachaud}\par}

\begin{abstract}
We prove lower bounds for the minimum distance of algebraic geometry 
codes over surfaces whose canonical divisor is either nef or anti-strictly nef and over surfaces without irreducible curves of small genus. We sharpen these lower bounds for surfaces whose arithmetic Picard number equals one, surfaces without curves with small self-intersection and fibered surfaces. Finally we specify our bounds to the case of surfaces of degree $d\geq 3$ embedded in $\mathbb{P}^3$.
\end{abstract}

\setcounter{tocdepth}{1}
\tableofcontents

\section{Introduction}\label{Intro}

The construction of Goppa codes over algebraic curves (\cite{Goppa}) 
has enabled Tsfaman, Vl\u{a}du\c{t} and Zink to beat the Gilbert-Varshamov bound (\cite{Tsfasman_Vladut_Zink}).
Since then, algebraic geometry codes over curves have been largely studied.
Even though the same construction holds on varieties of higher dimension,
 the literature is less abundant in this context. However one can consult \cite{LittleHigher} for a survey of Little and \cite{soH}
for an extensive use of intersection theory involving the Seshadri constant proposed by S. H. Hansen.
Some work has also been undertaken
in the direction of surfaces.
Rational surfaces yielding to good codes were constructed by Couvreur in \cite{Couvreur} from some blow-ups of the plane and by Blache {\it et al.} in \cite{Blache} from Del Pezzo surfaces. 
Codes from cubic surfaces where studied by Voloch and Zarzar in~\cite{Voloch_Zarzar}, from toric surfaces by J. P. Hansen in~\cite{hansen}, from Hirzebruch surfaces by Nardi in \cite{Nardi}, from ruled surfaces by one of the authors in  \cite{Aubry} and from abelian surfaces by Haloui in \cite{Haloui} in the specific case of simple Jacobians of genus $2$ curves, and by the authors in \cite{Aubry_Berardini_Herbaut_Perret} for general abelian surfaces. Furthermore Voloch and Zarzar (\cite{Voloch_Zarzar}, \cite{Zarzar}) and Little and Schenck (\cite{Little_Schenck}) have studied  surfaces whose arithmetic Picard number is one.

\medskip

The aim of this paper is to provide a study of the minimum distance $d(X, rH, S)$ of the algebraic geometry code ${\mathcal C}(X, rH, S)$ constructed  from an algebraic surface $X$, a set $S$ of rational points on $X$, a rational effective ample divisor $H$ on $X$ avoiding $S$ and an integer $r>0$.

\medskip

We prove in Section~\ref{Codes_over_algebraic_surfaces} lower bounds for the minimum distance $d(X, rH, S)$ under some specific assumptions on the \emph{geometry of the surface} itself. 
Two quite wide families of surfaces are studied. The first one is that of surfaces whose canonical divisor is either nef or anti-strictly nef. The second one consists of surfaces which do not contain irreducible curves of low genus. We obtain the following theorem, where we denote, as in the whole paper, the finite field with $q$ elements by $\mathbb{F}_q$ and the virtual arithmetic genus of a divisor $D$ by $\pi_D$, and where we set $m:=\lfloor 2\sqrt{q}\rfloor$.

\begin{thm}(Theorem \ref{ourbound} and Theorem \ref{theorem_ell})
Let $X$ be an absolutely irreducible smooth projective algebraic surface 
defined over $\mathbb{F}_q$ whose canonical divisor is denoted by $K_X$.
Consider a set $S$ of rational points on $X$, 
  a rational effective ample divisor $H$  avoiding $S$, and  a positive integer $r$.
 In order to compare the following bounds, we set
$$
d^*(X, rH, S)\coloneqq \sharp S-rH^2(q+1+m)-m(\pi_{rH}-1).
$$

\begin{itemize}
\item[1)]
\begin{itemize}
\item[(i)] If $K_X$ is nef, then 
$$d(X, rH, S) \geq d^*(X, rH, S).$$
\item[(ii)] If $-K_X$ is strictly nef, then 
$$d(X, rH, S) \geq d^*(X, rH, S)+mr(\pi_{H}-1).$$ 
\end{itemize}
\item[2)]
If
there exists an integer
$\ell >0$ 
such that any ${\mathbb F}_q$-irreducible curve lying on $X$ 
and defined over ${\mathbb F}_q$ has arithmetic genus strictly greater than $\ell$, then
$$d(X, rH, S) \geq d^*(X, rH, S)+\left(rH^2-\frac{\pi_{rH}-1}{\ell}\right)(q+1+m).$$
\end{itemize}
\end{thm}
Inside both families, adding some extra geometric assumptions on the surface yields in Section~\ref{Improvements} to some improvements for these lower bounds. 
This is the case for surfaces whose arithmetic Picard number is one,
for  surfaces without irreducible curves defined over ${\mathbb F}_q$ with small self-intersection, so as for fibered surfaces. 
Theorems~\ref{theorem_fibration} and~\ref{theorem_fibration_two} 
(that hold for fibered surfaces)
improve the bounds of 
Theorems~\ref{ourbound} and~\ref{theorem_ell}
(that hold for the whole wide families).
Indeed the bound on the minimum distance $d(X, rH, S)$
is increased by
the non-negative \emph{defect} 
$\delta(B) = q+1+mg_B-\sharp B({\mathbb F}_q)$ of the base curve $B$.
Finally in Section \ref{Hypersurfaces} we specify our bounds to the case of surfaces of degree $d\geq 3$ embedded in $\mathbb{P}^3$.

\medskip

Characterizing surfaces that yield good codes seems to be a complex question.
It is not the goal of our paper to produce good codes: we aim to give theoretical bounds on the minimum distance of algebraic geometry codes on
general surfaces.
However one can derive from our work one or two  heuristics.
Indeed, Theorem \ref{theorem_ell} suggests to look for surfaces with no curves of small genus and
fibered surfaces provide natural examples of such surfaces (see Theorem \ref{theorem_fibration_two}).

\section{Background}\label{Backgrounds}
Codes from algebraic surfaces are defined
in the same way as on algebraic curves:
we evaluate some functions with prescribed poles on some sets of rational points.
Whereas the key tool for the study of the minimum distance 
in the $1$-dimensional case is the mere fact
that a function has as many zeroes as poles, 
in the $2$-dimensional case most of the proofs rest on intersection theory.

We sum up in this section the few results on intersection theory we need. 
Following the authors cited in the Introduction 
we recall the definition of the algebraic geometry codes.
We also recall quickly how the dimension 
of the code can be bounded from below
 under the assumption of the injectivity 
 of the evaluation map. 
Then we prove a lemma that will be used in the course of the paper to bound from below 
 the minimum distance of the code for several families of surfaces.
 Finally, we recall some results on the number of rational points on curves over finite fields.

\subsection{Intersection theory} \label{Intersection_theory}
Intersection theory has almost become a mainstream tool 
to study codes over surfaces 
(see \cite{Aubry}, \cite{soH}, \cite{Voloch_Zarzar}, \cite{Zarzar},
 \cite{Little_Schenck}, \cite{Aubry_Berardini_Herbaut_Perret})
and it is also central in our proofs.
We do not recall here the classical definitions 
of the different equivalent classes of divisors
and we refer the reader to \cite[\S V]{Hartshorne} for a presentation. 
We denote by $\NS(X)$
the \emph{arithmetic} N\'eron-Severi group of a smooth surface $X$ defined over ${\mathbb{F}_q}$
whose rank is called
the arithmetic Picard number of $X$,
or Picard number for short.
Recall that a divisor $D$ on $X$ is said to be \emph{nef} (respectively  \emph{strictly nef}) if $D.C\geq 0$ (respectively $D.C>0$) for any irreducible curve $C$ on $X$. 
A divisor $D$ is said to be \emph{anti-ample} if $-D$ is ample, \emph{anti-nef} if $-D$ is nef and \emph{anti-strictly nef} if $-D$ is strictly nef.
Let us emphasize three classical results we will use in this paper.

The first one is (a generalisation of) the adjunction formula
 (see~\cite[\S V, Exercise 1.3]{Hartshorne}).  For any ${\mathbb F}_q$-irreducible curve $D$ on $X$ of arithmetic genus $\pi_D$,
 we have
\begin{equation}\label{adjunction_formula}
D.(D+K_X)=2\pi_D-2
\end{equation}
where  $K_X$ is the canonical divisor on $X$.
This formula allows us to define the virtual arithmetic genus of any divisor $D$ on $X$.

The second one is the corollary of the Hodge index theorem stating that 
if $H$ and $D$ are two divisors on $X$
with $H$ ample, then 
\begin{equation}\label{Hodge} 
H^2D^2 \leq (H.D)^2,
\end{equation}
where equality holds if and only if $H$ and $D$ are numerically proportional.

The last one is a simple outcome of B\'ezout's theorem in projective spaces (and the trivial part of the Nakai-Moishezon criterion).
It ensures  that for any ample divisor $H$ on $X$ and for any irreducible curve $C$ on $X$, we have
$H^2>0$ and  $H.C>0$.

\subsection{Algebraic geometry codes} \label{section_eval_code}
\subsubsection{Definition of AG  codes} We study, as in the non-exhaustive list of papers 
\cite{Aubry}, \cite{Voloch_Zarzar},
\cite{Couvreur}, \cite{soH}, \cite{Zarzar}, \cite{Haloui}, \cite{Little_Schenck} and \cite{Aubry_Berardini_Herbaut_Perret},
the generalisation of Goppa 
algebraic geometry codes from curves to surfaces.
In the whole paper 
we consider an absolutely irreducible smooth projective algebraic surface $X$ 
defined over $\mathbb{F}_q$ and a set $S$ of rational points on $X$. 
Given a rational effective ample divisor $G$ on $X$ avoiding $S$,
the algebraic geometry code, or AG code for short,
is defined 
by evaluating the elements of
the Riemann-Roch space $L(G)$
at the points of $S$.
Precisely we define the linear code 
 $\mathcal{C} (X,G, S)$ as the image of the evaluation map 
$\ev : L(G) \longrightarrow \mathbb{F}^{\sharp S}_q$.

\subsubsection{Length and dimension of AG codes} From the very definition, 
the length of the code is  $\sharp S$.
As soon as the morphism $\ev$ is injective - see~(\ref{injsi>0}) for a sufficient condition -
the dimension of the code equals $\ell(G)=\dim_{\mathbb{F}_q}L(G)$ which can be easily bounded from below using standard algebraic geometry tools as follows.
By Riemann-Roch theorem (see \cite[V, \S 1]{Hartshorne}), we have
$$\ell(G)-s(G)+\ell (K_X-G)=\frac{1}{2}G.(G-K_X)+1+p_a(X)$$
where $p_a(X)$ is the arithmetic genus of $X$, and where the so-called {\sl superabundance}
 $s(G)$ of $G$ in $X$ is non-negative (as it is the dimension of some vector space).
Now, under the assumption that
\begin{equation}\label{condition_dim}
K_X.A< G.A,
\end{equation}
for some ample divisor $A$, we have from \cite[V, Lemma 1.7]{Hartshorne} that $\ell (K_X-G)=0$. 
Thus, if the evaluation map $\ev$ is injective and under assumption~(\ref{condition_dim}), we get the lower bound
\begin{equation}\label{lowerbound_dim}
\dim\mathcal{C}(X,G, S)=\ell(G)\geq \frac{1}{2}G.(G-K_X)+1+p_a(X)
\end{equation}
 for the dimension of the code $\mathcal{C}(X,G, S)$.

\subsubsection{Toward the minimum distance of AG codes} \label{toward_d} It follows that the difficulty lies in the estimation of the minimum distance $d(X,G, S)$ of the code.
For any non-zero $f \in L(G)$, we introduce 
 the number $N(f)$ of rational points of the divisor of zeroes of $f$. 
The Hamming weight  $w(\ev(f))$ of the codeword $\ev(f)$ satisfies
 \begin{equation}\label{weight}
w(\ev(f)) \geq \sharp S -N(f),
 \end{equation}
 from which it follows that
\begin{equation} \label{d>=}
d(X,G, S) \geq \sharp S-\max_{f\in L(G)\setminus \{0\}} N(f).
\end{equation}
We also deduce from~(\ref{weight}) that
\begin{equation} \label{injsi>0}
\ev \hbox{~is injective if~} \max_{f\in L(G)\setminus \{0\}} N(f) < \sharp S.
 \end{equation}

We now broadly follow the way of \cite{Haloui}. We associate to any non-zero
function $f\in L(G)$ the rational effective  divisor
\begin{equation}\label{decomposition}
D_f:=G+(f)=\sum_{i=1}^k n_i D_i \geq 0,
\end{equation}
where $(f)$ is the principal divisor defined by $f$,
the $n_i$ are positive
integers
 and each $D_i$ is a reduced ${\mathbb F}_q$-irreducible curve.

Note that in this setting, the integer $k$ and the curves $D_i$'s depend on $f \in L(G)$. Several lower bounds for the minimum distance $d(X, G, S)$ in this paper will follow from the key lemma below.

\begin{lemma}\label{lemma_d}
Let $X$ be a smooth projective surface defined over $\mathbb{F}_q$, let $S$ be a set of rational points on $X$ and let $G$ be a rational effective  divisor on $X$ avoiding $S$.
Set $m = \lfloor 2\sqrt{q}\rfloor$ and  keep the  notations introduced in~(\ref{decomposition}).
If there exist non-negative real numbers $a, b_1, b_2, c$, such that for any non-zero $f \in L(G)$ 
the three following assumptions are satisfied
\begin{enumerate}
\item\label{first} $k \leq a$,
\item\label{second} $\sum_{i=1}^k \pi_{D_i} \leq b_1+kb_2$ \textrm{ and }
\item\label{third} for any $1\leq i \leq k$ we have $ \sharp D_i(\mathbb{F}_q) \leq c+m\pi_{D_i}$
\end{enumerate}
then the minimum distance $d(X, G, S)$ of ${\mathcal C}(X, G, S)$ satisfies
$$d(X,G,S) \geq \sharp S-a(c+mb_2)-mb_1.$$
\end{lemma}

\begin{proof}
Let us write the principal divisor $(f)=(f)_0-(f)_{\infty}$ as the difference of its effective divisor of zeroes minus its effective divisor of poles. Since $G$ is effective and $f$ belongs to $L(G)$, we have $(f)_{\infty} \leq G$. Hence, formula~(\ref{decomposition}) reads
$G+(f)_0-(f)_{\infty}=\sum_{i=1}^k n_i D_i$,
that is
$$(f)_0 = \sum_{i=1}^k n_iD_i +(f)_{\infty}-G \leq \sum_{i=1}^k n_iD_i.$$
This means that any ${\mathbb F}_q$-rational point of $(f)_0$ lies in some $D_i$ so
\begin{equation} \label{N(f)<=}
N(f)   \leq \sum_{i=1}^k \sharp D_i(\mathbb{F}_q).
\end{equation}
Then it follows successively from the assumptions of the lemma that
$$N(f) \leq \sum_{i=1}^k (c+m\pi_{D_i})\leq kc+m(b_1+kb_2)\leq mb_1+a(c+mb_2).$$
Finally Lemma~\ref{lemma_d} follows from~(\ref{d>=}).
\end{proof}

\begin{remark}
In several papers,
the point of departure to estimate
the minimum distance is 
a bound on the number of components $k$, 
which corresponds to the condition (\ref{first}) of  Lemma 
\ref{lemma_d} above.
In the special case where $\NS(S)=\mathbb{Z}H$ and $G=rH$, for $H$ an ample divisor on $X$, 
Voloch and Zarzar have proven in \cite{Voloch_Zarzar} that $k \leq r$.
In the present paper we obtain a bound on $k$ in a more general context,
that is when the N\'eron-Severi group
has rank greater than one 
(see for example  Lemma \ref{k_HodgeTGV_bounds}, 
point (\ref{k_bound_ell}) of Lemma \ref{lemma_ell}
and point (\ref{k_beta}) of Lemma \ref{lemma_quadra}).
\end{remark}

\subsection{Two upper bounds for the number of rational points on curves} 

We manage to fulfill assumption~(\ref{third}) in Lemma~\ref{lemma_d} using the bounds on the number of rational points given in
Theorem~\ref{irreduciblecurves} and Proposition~\ref{covering} below.
Point $(\ref{L-S})$ of Theorem~\ref{irreduciblecurves} appears in the proof of Theorem 3.3 of Little and Schenck in \cite{Little_Schenck} within a more restrictive context, whereas point $(\ref{A-P})$  follows from \cite{Aubry_Perret_Weil}.
We state a general theorem and give here the full proof for the sake of completeness following \cite{Little_Schenck}.

\begin{theorem}[Aubry-Perret~\cite{Aubry_Perret_Weil} and Little-Schenck~\cite{Little_Schenck}] \label{irreduciblecurves}
Let $D$ be an ${\mathbb F}_q$-irreducible curve of arithmetic genus $\pi_D$ lying on a smooth projective algebraic surface. Then,
\begin{enumerate}
\item\label{A-P} we have $\sharp D({\mathbb F}_q)  \leq q+1+m\pi_D$.
\item\label{L-S}(Little-Schenck) If moreover $D$ is \emph{not} absolutely irreducible, we have 
$$\sharp D({\mathbb F}_q) \leq \pi_D+1.$$
\end{enumerate}
\end{theorem}

\begin{proof}
We first prove the second item, following the proof of  \cite[Th. 3.3]{Little_Schenck}. Since $D$ is ${\mathbb F}_q$-irreducible, the Galois group $\Gal(\overline{\mathbb{F}}_q/\mathbb{F}_q)$ acts transitively on the set of its ${\bar{r}} \geq 1$ absolutely irreducible components $D_1,\dots,D_{\bar{r}}$. Since a ${\mathbb F}_q$-rational point on $D$ is stable under the action of $\Gal(\overline{\mathbb{F}}_q/\mathbb{F}_q)$, it lies in the intersection $\cap_{1\leq i\leq \bar{r}}D_i$. Under the assumption that $D$ is not absolutely irreducible, that is ${\bar{r}}\geq 2$, it follows that $\sharp D(\mathbb{F}_q)\leq \sharp (D_i\cap D_j)(\overline{\mathbb{F}}_q)\leq D_i.D_j$  for every couple $(i,j)$ with $i\neq j$.

As a divisor, $D$ can be written over $\overline{{\mathbb F}}_q$ as $D=\sum_{i=1}^{\bar{r}} a_iD_i$. 
By transitivity of the Galois action, we have $a_1=\cdots = a_{\bar{r}}=a$. Now since $D$ can be assumed to be reduced, we have $a=1$, so that finally $D=\sum_{i=1}^{\bar{r}} D_i$. Using the adjonction formula (\ref{adjunction_formula}) for $D$ and each $D_i$, and taking into account that $\pi_{D_i} \geq 0$ for any $i$, we get
\begin{align*}
2\pi_D-2&=(K_X+D).D\\
&=\sum_{i=1}^{\bar{r}} (K_X+D_i).D_i+\sum_{i\neq j} D_i.D_j\\
&=\sum_{i=1}^{\bar{r}}(2\pi_{D_i}-2)+\sum_{i\neq j} D_i.D_j\\
&\geq -2\bar{r}+\sum_{i\neq j} D_i.D_j.
\end{align*}
Since there are $\bar{r}(\bar{r}-1)$ pairs $(i, j)$ with $i\neq j$, we deduce that for at least one such pair $(i_0, j_0)$, we have
$$D_{i_0}.D_{j_0} \leq \frac{2(\pi_D-1+\bar{r})}{\bar{r}(\bar{r}-1)}.$$
It is then easily checked that the left hand side of the former inequality is a decreasing function of $\bar{r} \geq 2$, so that we obtain
$$\sharp D(\mathbb{F}_q)\leq D_{i_0}.D_{j_0} \leq \frac{2(\pi_D-1+2)}{2(2-1)}=\pi_D+1$$
and the second item is proved.

The first item follows  from Aubry-Perret's bound in~\cite{Aubry_Perret_Weil} in case $D$ is absolutely irreducible, that is in case $\bar{r}=1$, and from the second item in case $D$ is not absolutely irreducible since $\pi_D+1\leq q+1+m\pi_D$.
\end{proof}

The following bound will be usefull in Subsection \ref{Section_fibration} for the study of codes from fibered surfaces.

\begin{proposition}[Aubry-Perret~\cite{Aubry_Perret_FFA}] \label{covering} Let $C$ be a smooth projective absolutely irreducible curve of genus $g_C$ over ${\mathbb F}_q$ and $D$ be an ${\mathbb F}_q$-irreducible curve having $\bar{r}$ absolutely irreducible components $\overline{D}_1,\dots,\overline{D}_{\bar{r}}$. Suppose there exists a regular map 
 $D \rightarrow C$ 
 in which none absolutely irreducible component maps onto a point.
Then
$$|\sharp D(\mathbb{F}_q)-\sharp C(\mathbb{F}_q)|\leq (\overline{r}-1)q+m(\pi_D-g_C).$$
\end{proposition}

\begin{proof}
Since $C$ is smooth and none geometric component of $D$ maps onto a point, the map $D\rightarrow C$ is flat. Hence by~\cite[Th.14]{Aubry_Perret_FFA} we have
$$|\sharp D(\mathbb{F}_q)-\sharp C(\mathbb{F}_q)|\leq (\overline{r}-1)(q-1)+m\left(\sum_{i=1}^{\overline{r}} g_{\overline{D}_i} - g_C\right)+\Delta_D$$
where $\Delta_D=\sharp \tilde{D}(\overline{\mathbb{F}}_q)-\sharp D(\overline{\mathbb{F}}_q)$ with
$\tilde{D}$  the normalization of $D$.
The result follows from \cite[Lemma 2]{Aubry_Perret_FFA} 
where it is proved that $m\sum_{i=1}^{\overline{r}} g_{\overline{D}_i}+\Delta_D-\bar{r}+1\leq m\pi_D$.
\end{proof}

\section{The minimum distance of codes over some families of algebraic surfaces}\label{Codes_over_algebraic_surfaces}
We are unfortunately unable to fulfill simultaneously assumptions (\ref{first}) and (\ref{second}) of Lemma~\ref{lemma_d}
for general surfaces. So we focus on two families of algebraic surfaces where we do succeed. To begin with, let us fix some common notations. 

We consider a rational effective ample divisor $H$ on the surface $X$ avoiding a set $S$ of rational points on $X$ and for a positive integer $r$ we consider $G=rH$. 
We study, in accordance to Section~\ref{section_eval_code}, the evaluation code $\mathcal{C}(X,rH, S)$ and we denote by $d(X, rH, S)$ its minimum distance.

\subsection{Surfaces whose canonical divisor is either nef or anti-strictly nef}\label{nef}
We study in this section codes defined over surfaces such that either the canonical divisor $K_X$ is nef, or its opposite $-K_X$ is strictly nef.
This family is quite large. It contains, for instance:
\begin{itemize}
\item[-] surfaces whose canonical divisor $K_X$ is anti-ample.
\item[-] Minimal surfaces of Kodaira dimension $0$, for which the canonical divisor is numerically zero, hence nef. These are abelian surfaces, $K3$ surfaces, Enriques surfaces and hyperelliptic or quasi-hyperelliptic surfaces (see \cite{Bombieri_Mumford}).
\item[-] Minimal surfaces of Kodaira dimension $2$. These are the so called minimal surfaces of general type. For instance, surfaces in ${\mathbb P}^3$ of degree $d\geq 4$, without curves $C$ with $C^2=-1$, are minimal of general type.
\item[-] Surfaces whose arithmetic Picard number is one.
\item[-] Surfaces of degree $3$ embedded in $\mathbb{P}^3$.
\end{itemize}

The main theorem of this section (Theorem~\ref{ourbound}) rests mainly on the next lemma
designed to fulfill assumptions (\ref{first}) and (\ref{second}) of Lemma~\ref{lemma_d}.

\begin{lemma}\label{k_HodgeTGV_bounds}
Let $D=\sum^k_{i=1} n_iD_i$
be the decomposition as a sum of ${\mathbb F}_q$-irreducible and reduced curves of an effective divisor $D$ linearly equivalent to $rH$.
Then we have:

\begin{enumerate}
\item \label{k_bound} $k\leq rH^2;$
\item \label{pointtwo}
\begin{enumerate}[label=(\roman*)]
\item \label{if_nef} if $K_X$ is nef, then $\sum_{i=1}^k \pi_{D_i} \leq \pi_{rH}-1+k$;
\item \label{not_nef} if $-K_X$ is strictly nef,
then $\sum_{i=1}^k \pi_{D_i} \leq \pi_{rH}-1 -\frac{1}{2}rH.K_X+\frac{1}{2}k$.
\end{enumerate}
\end{enumerate}
\end{lemma}

\begin{proof}
Using that $D$ is numerically equivalent to $rH$, that $n_i>0$ and $D_i.H>0$ for every $i=1,\dots,k$  since $H$ is ample, 
we prove item~$(\ref{k_bound})$:
$$rH.H=D.H=\sum^k_{i=1} n_iD_i.H\geq \sum^k_{i=1} D_i.H\geq k.$$

Now we apply inequality (\ref{Hodge}) to $H$ and $D_i$ for every $i$, to get
$D_i^2H^2\leq(D_i.H)^2$.
We thus have, together with adjunction formula (\ref{adjunction_formula}) and 
inequality $H^2> 0$,
\begin{equation}\label{sum_pi}
\pi_{D_i}-1\leq (D_i.H)^2/2H^2+D_i.K_X/2.
\end{equation}

 To prove point \ref{if_nef} of item (\ref{pointtwo})  
 we sum from $i=1$ to $k$ and thus obtain 
\begin{equation}\label{demo}
\begin{split}
\sum_{i=1}^k \pi_{D_i}-k &\leq \frac{1}{2H^2}\sum_{i=1}^k(D_i.H)^2+\frac{1}{2}\sum_{i=1}^k D_i.K_X\\
&\leq \frac{1}{2H^2}\left(\sum_{i=1}^k n_iD_i.H \right)^2+\frac{1}{2}\sum_{i=1}^k n_iD_i.K_X\\
&\leq \frac{(rH.H)^2}{2H^2}+\frac{rH.K_X}{2}\\
&=\pi_{rH}-1,
\end{split}
\end{equation}
where we use the 
positivity of the coefficients $n_i$,
the numeric equivalence between $D$ and $\sum_{i=1}^k n_i D_i$,
the fact that $H$ is ample and the hypothesis taken on $K_X$.

Under the hypothesis of point \ref{not_nef} we have $D_i.K_X \leq -1$.
Replacing in the first line of (\ref{demo}) gives
$\sum_{i=1}^k \pi_{D_i}-k \leq \frac{1}{2H^2}\sum_{i=1}^k(D_i.H)^2 -\frac{k}{2} $.
We conclude in the same way.
\end{proof}

\begin{theorem}\label{ourbound}
Let $H$ be a rational effective ample divisor on a surface $X$ avoiding a set $S$ of rational points and let $r$ be a positive integer.
We set
\begin{equation}\label{d*}
d^*(X, rH, S)\coloneqq \sharp S-rH^2(q+1+m)-m(\pi_{rH}-1).
\end{equation}
\begin{itemize}
\item[(i)] If $K_X$ is nef, then 
$$d(X, rH, S) \geq d^*(X, rH, S).$$
\item[(ii)] If $-K_X$ is strictly nef, then 
$$d(X, rH, S) \geq d^*(X, rH, S)+mr(\pi_{H}-1).$$ 
\end{itemize}
\end{theorem}

\begin{proof}
The theorem follows from Lemma~\ref{lemma_d} for which assumptions (\ref{first}) and (\ref{second}) hold from Lemma~\ref{k_HodgeTGV_bounds} and assumption (\ref{third}) holds from Theorem~\ref{irreduciblecurves}.
\end{proof}

\subsection{Surfaces without irreducible curves of small genus}\label{Sec_without_small_genus_curves}
We consider in this section surfaces $X$ with the property that there exists an integer $\ell \geq 1$ such that any ${\mathbb F}_q$-irreducible curve $D$ lying on $X$ and defined over ${\mathbb F}_q$ has arithmetic genus $\pi_D\geq \ell +1$.
It turns out that under this hypothesis, we can  fulfill assumptions ~(\ref{first}) and (\ref{second}) of Lemma~\ref{lemma_d} without any hypothesis on $K_X$ 
contrary to the setting of Section~\ref{nef}.

Examples of surfaces with this property do exist. For instance:
\begin{itemize}
\item[-] simple abelian surfaces satisfy this property for $\ell =1$ (see~\cite{Aubry_Berardini_Herbaut_Perret} for abelian surfaces with this property for $\ell = 2$). 
\item[-] Fibered surfaces on a smooth base curve $B$ of genus $g_B\geq 1$ and generic fiber of arithmetic genus $\pi_0 \geq 1$, and whose singular fibers are ${\mathbb F}_q$-irreducible, do satisfy this property for $\ell = \min(g_B, \pi_0)-1$.
\item[-] Smooth surfaces in ${\mathbb P}^3$ of degree $d$ whose arithmetic Picard group is generated by the class of an hyperplane section do satisfy this property for $\ell = \frac{(d-1)(d-2)}{2}-1$ (see Lemma \ref{genus_on_hypersurfaces}).
\end{itemize}

\begin{lemma}\label{lemma_ell}
Let $X$ be a surface without ${\mathbb F}_q$-irreducible curves of arithmetic genus less than or equal to $\ell$ for $\ell$ a positive integer. 
Consider a rational effective ample divisor $H$ on $X$ and a positive integer $r$.
Let $D=\sum^k_{i=1} n_iD_i$ be the decomposition as a sum of ${\mathbb F}_q$-irreducible and reduced curves of an effective divisor $D$ linearly equivalent to $rH$. Then we have
\begin{enumerate}
\item\label{k_bound_ell} $k\leq \frac{\pi_{rH}-1}{\ell}$;
\item\label{sum_pi_i_ell}  $\sum_{i=1}^k \pi_{D_i} \leq \pi_{rH}-1+k$.
\end{enumerate}
\end{lemma}

In case $X$ falls in both families of Section~\ref{nef} and of this Section~\ref{Sec_without_small_genus_curves}, the present new bound of the first item for $k$ is better than the one of Lemma~\ref{k_HodgeTGV_bounds} if and only if $\pi_{rH}-1<\ell rH^2$, that is if and only if $\ell>\frac{H.K_X}{2H^2}+\frac{r}{2}$.
In the general setting, this inequality sometimes holds true, sometimes not. As a matter of example, supposed $K_X$ to be ample and let us consider $H=K_X$. In this setting the inequality holds  if and only if  $r< 2\ell - 1$.

\begin{proof} By assumption, we have $0\leq \ell\leq \pi_{D_i}-1$ and $n_i \geq 1$ for any $1 \leq i \leq k$, hence using adjunction formula (\ref{adjunction_formula}), we have
$$2\ell k \leq 2\sum_{i=1}^k (\pi_{D_i}-1)
\leq 2\sum_{i=1}^k n_i(\pi_{D_i}-1)
= \sum_{i=1}^k n_i D_i^2+\sum_{i=1}^k n_iD_i.K_X.$$
Moreover using (\ref{Hodge}) and (\ref{decomposition}), we get
$$2\ell k 
\leq \sum_{i=1}^k n_i \frac{(D_i.H)^2}{H^2}+\left(\sum_{i=1}^k n_iD_i\right).K_X
\leq \sum_{i=1}^k n_i^2 \frac{(D_i.H)^2}{H^2}+rH.K_X.$$
Since $H$ is ample, we obtain
$$2\ell k 
\leq \sum_{i, j=1}^k n_i n_j \frac{(D_i.H)(D_j.H)}{H^2}+rH.K_X
=  \frac{(\sum_{i,=1}^k n_i D_i.H)^2}{H^2}+rH.K_X.$$
By (\ref{decomposition}), we conclude that
$$2\ell k\leq
\frac{(rH.H)^2}{H^2}+rH.K_X
=2(\pi_{rH}-1),$$
and both items of Lemma~\ref{lemma_ell} follow.
\end{proof}

\begin{theorem}\label{theorem_ell}
Let $X$ be a surface without ${\mathbb F}_q$-irreducible curves of arithmetic genus less than or equal to $\ell$
 for $\ell$ a positive integer. Consider a rational effective ample divisor $H$ on $X$ avoiding a finite set $S$ of rational points and let $r$ be a positive integer.
 Then we have
$$d(X, rH, S) \geq d^*(X, rH, S)+\left(rH^2-\frac{\pi_{rH}-1}{\ell}\right)(q+1+m).$$
\end{theorem}

\begin{proof}
The theorem follows from Lemma~\ref{lemma_d}, for which items (\ref{first}) and (\ref{second}) hold from Lemma~\ref{lemma_ell} and item (\ref{third}) holds from Theorem~\ref{irreduciblecurves}.
\end{proof}

\section{Four improvements}\label{Improvements}
In this section we manage to obtain better parameters for conditions
(\ref{first}), (\ref{second}) or (\ref{third}) of Lemma~\ref{lemma_d} 
in four
 cases:
for surfaces of arithmetic Picard number one, 
for surfaces which do not contain $\mathbb{F}_q$-irreducible curves
of small self-intersection and
whose canonical divisor is either nef or anti-nef, 
for fibered surfaces with nef canonical divisor, and for fibered surfaces whose singular fibers
are $\mathbb{F}_q$-irreducible curves.

\subsection{Surfaces with Picard number one}\label{Picard_1}
As mentioned in the Introduction,
the case of surfaces $X$ whose arithmetic Picard number equals one has already attracted 
some interest (see \cite{Zarzar}, \cite{Voloch_Zarzar}, \cite{Little_Schenck} and \cite{Blache}).
We prove in this subsection Lemma~\ref{bound_picard_rank_one} and Theorem~\ref{theorem_picard_number_one} which improve, under this rank one assumption, the bounds of Lemma \ref{k_HodgeTGV_bounds} and Theorem \ref{ourbound}.
These new bounds depend on the sign of $3H^2+H.K_X$, where $H$ is the ample generator of $\NS(X)$.

\begin{lemma}\label{bound_picard_rank_one}
Let $X$ be a smooth projective surface 
of  arithmetic Picard number one. Let $H$ be the ample generator of $\NS(X)$ and let $r$ be a positive integer.
For any non-zero function $f \in L\left( rH \right)$ 
consider the decomposition
$D_f=\sum_{i=1}^k n_i D_i$ into
$\mathbb{F}_q$-irreducible and reduced curves $D_i$ with positive integer coefficients $n_i$ 
as in (\ref{decomposition}).
Then the sum of the arithmetic genera of the curves $D_i$ satisfies:
 \begin{itemize}
\item[(i)] $\sum_{i=1}^k{\pi}_{D_i}\leq 
(k-1)\pi_{H} + \pi_{(r-k+1)H}$ if $3H^2+H.K_X \geq 0$;
\item[(ii)] $\sum_{i=1}^k{\pi}_{D_i} \leq 
H^2(r-k)^2/2+H^2(r-2k)+k$ if  $3H^2+H.K_X <0$.
\end{itemize}
\end{lemma}\label{proposition:maxLr}

\begin{remark} Note that the condition $3H^2+H.K_X\geq 0$ is satisfied
as soon as $H.K_X\geq0$.
It is also satisfied in the special case where $K_X=-H$
which corresponds to Del Pezzo surfaces.

\end{remark}

\begin{proof}
In order to prove the first item, we consider a non-zero function $f \in L\left( rH \right)$ 
and we keep the notation already introduced in (\ref{decomposition}), namely
$D_f=\sum_{i=1}^k n_i D_i$.
As $\NS(X)=\mathbb{Z}H$, 
for all $i$ we have $D_i=a_i H$ and
we know by Lemma 2.2 in \cite{Zarzar}
that $k \leq r$.
Intersecting with the ample divisor $H$ enables to prove that for all $i$ 
we have $a_i \geq 1$
and that $\sum_{i=1}^k n_i a_i =r$. 
Thus to get an upper bound for 
$\sum_{i=1}^k \pi_{D_i}=\sum_{i=1}^k \pi_{a_i H}$,
we are reduced to bounding 
$ \left( \sum_{i=1}^k a_i^2 \right) H^2/2 + \left( \sum_{i=1}^k a_i \right) H.K_X/2 + k $
under the constraint $\sum_{i=1}^k a_i n_i =r$.
Our strategy is based on the two following arguments.

First, the condition $3H^2+H.K_X \geq 0$ guarantees
that $a \mapsto \pi_{aH}$ is an increasing sequence.
Indeed, for integers $a'>a\geq 1$
we have $\pi_{a' H} \geq \pi_{a H}$ if and only if
$(a+a')H^2 \geq -H.K_X$,
which is true under the condition above because $a+a' \geq 3$.
As a consequence, if we
fix an index $i$ between $1$ and $k$ and if we
consider that the product $n_i a_i$ is constant, 
then the value of $\pi_{a_i H}$ is maximum when $a_i$ is,
that is when $a_i=n_ia_i$ and $n_i=1$.

Secondly,  
assume that all the $n_i$ equal $1$
and that $\sum_{i=1}^k a_i=r$.
We are now reduced to bounding 
$  \sum_{i=1}^k a_i^2 $.
We can prove that the maximum is reached when all 
the $a_i$ equal $1$ except one which equals $r-k+1$.
Otherwise, suppose for example that
$2 \leq a_1 \leq a_2$. 
Then $a_1^2+a_2^2 < (a_1-1)^2 + (a_2+1)^2$
and $ \sum_{i=1}^k a_i^2$ is not maximum, and the first item is thus proved.

For the second item, using the adjonction formula we get
$$\sum_{i=1}^k \pi_{D_i}-k \leq \frac{1}{2H^2}\sum_{i=1}^k(D_i.H)^2+\frac{1}{2}\sum_{i=1}^k D_i.K_X.$$
Again as $\NS(X)=\mathbb{Z}H$, for all $i$ we have $D_i=a_iH$. Thus we get
$$\sum_{i=1}^k \pi_{D_i}-k \leq \frac{1}{2H^2}\sum_{i=1}^ka_i^2(H^2)^2+\frac{1}{2}\sum_{i=1}^k a_iH.K_X.$$
Now using that $H.K_X\leq -3H^2$ by hypothesis, that $\sum_{i=1}^k a_i \geq k$ since every $a_i$ is positive and that since $\sum_{i=1}^k a_i\leq r$ we can prove again that $\sum_{i=1}^k a_i^2\leq (r-k+1)^2+(k-1)$, we get
$$\sum_{i=1}^k \pi_{D_i}-k \leq \frac{H^2}{2}((r-k+1)^2+(k-1))-\frac{3H^2}{2}k.$$
Some easy calculation shows that this is equivalent to our second statement.
\end{proof}
 
\begin{theorem}\label{theorem_picard_number_one}
Let $X$ be a smooth projective surface of arithmetic Picard number one. Let $H$ be the ample generator of $\NS(X)$ and $S$ a finite set of rational points avoiding $H$.
For any positive integer $r$, the minimum distance $d(X, rH, S)$ of the code $\mathcal{C}(X,rH,S)$ satisfies:
\begin{enumerate}[label=(\roman*)]
\item\label{fab} if  $3H^2+H.K_X \geq 0$, then
$$d(X, rH,S)\geq
\begin{cases}
\sharp S-(q+1+m\pi_{rH}) \text{ if } r>2(q+1+m)/mH^2,\\
\sharp S -r(q+1+m\pi_{H}) \text{ otherwise. } 
\end{cases}
$$
\item\label{ele} If  $3H^2+H.K_X < 0$, then
$$d(X, rH, S)\geq
\begin{cases}
\sharp S -(q+1+m)-mH^2(r^2-3)/2 \text{ if } r>2(q+1+m)/mH^2-3,\\
\sharp S -r(q+1+m-mH^2) \text{ otherwise. }
\end{cases}
$$
\end{enumerate}
\end{theorem}
\begin{proof}
For any non-zero $f \in L(rH)$, we have by (\ref{N(f)<=}) and 
by point $(\ref{A-P})$  of Theorem \ref{irreduciblecurves} the following inequality
 $$N(f) \leq k(q+1)+m\sum_{i=1}^k \pi_{D_i}.$$
 We apply Lemma~\ref{bound_picard_rank_one}  to bound $\sum_{i=1}^k \pi_{D_i}$.
We get in the first case $N(f)\leq \phi(k)$ where $\phi(k):=m\pi_{(r-k+1)H}+k(q+1+m\pi_H)-m\pi_H.$
Remark that $\pi_{(r-k+1)H}$ is quadratic in $k$ and so $\phi(k)$ is a quadratic function with positive leading coefficient. In \cite[Lemma 2.2]{Voloch_Zarzar} Voloch and Zarzar proved that if $X$ has arithmetic Picard number one then
 $k\leq r$. Thus $\phi(k)$ attends its maximum for $k=1$ or for $k=r$ and $N(f)\leq \max \{\phi(1),\phi(r)\}$. A simple calculus shows that $\phi(1)-\phi(r)>0$ if and only if $r>2(q+1+m)/mH^2$. Since we have $d(X, rH, S) \geq \sharp S-\max_{f\in L(rH)\setminus \{0\}} N(f)$, part \ref{fab} of the theorem is proved.

The treatment of part \ref{ele} is the same, except that we use Lemma~\ref{bound_picard_rank_one} to bound $\sum_{i=1}^k \pi_{D_i}$.

\end{proof}

\begin{remark}\label{Little_Schenck}
Little and Schenck have given bounds in \cite[\S 3]{Little_Schenck} for the minimum distance of codes defined over algebraic surfaces of Picard number one. In particular, they obtain (if we keep the notations of Theorem \ref{theorem_picard_number_one}): $d(X, rH, S)\geq \sharp S-(q+1+m\pi_H)$  for $r=1$ (\cite[Th. 3.3]{Little_Schenck}) and $d(X, rH, S)\geq \sharp S -r(q+1+m\pi_H)$ for $r>1$ and $q$ large (\cite[Th. 3.5]{Little_Schenck}). Comparing their bounds with Theorem \ref{theorem_picard_number_one}, one can see that when $3H^2+H.K_X \geq 0$ we get the same bound for $r=1$ and also for $r>1$ without any hypothesis on $q$. Moreover, when $3H^2+H.K_X < 0$, our bounds improve the ones given by Little and Schenck, 
again without assuming large enough $q$ when $r>1$.
\end{remark}

\subsection{Surfaces without irreducible curves defined over ${\mathbb F}_q$ with small self-intersection and whose canonical divisor is either nef or anti-nef}\label{low_selfintersections}
We consider in this section surfaces $X$ such that 
 there exists some integer $\beta \geq 0$ for which any ${\mathbb F}_q$-irreducible curve $D$ lying on $X$ and defined over ${\mathbb F}_q$ has self-intersection $D^2\geq \beta$. We prove in this case Lemma~\ref{lemma_quadra} below, from which we can tackle assumption~(\ref{first}) in Lemma~\ref{lemma_d} in case $\beta >0$.
 Unfortunately, Lemma~\ref{lemma_quadra} enables to fulfill assumption~(\ref{second}) of Lemma~\ref{lemma_d} 
 only in case the intersection of the canonical divisor with ${\mathbb F}_q$-irreducible curves 
 has constant sign, that is for surfaces of Section~\ref{nef}.
The lower bound for the minimum distance we get is  better than the one given in Theorem~\ref{ourbound}.

Let us propose some examples of surfaces with this property:
\begin{itemize}
\item[-]  simple abelian surfaces satisfy this property for $\beta =2$. 
\item[-] Surfaces whose arithmetic Picard number is one. Indeed consider a curve $D$ defined over ${\mathbb F}_q$ on $X$, 
and assume that $\NS(X)={\mathbb Z}H$ with $H$ ample.
Then we have that $D=aH$ for some integer $a$. Since $H$ is ample we get 
 $1\leq D.H=aH^2$ hence $a\geq 1$ and $D^2=a^2H^2\geq H^2$. 
\item[-] Surfaces whose canonical divisor is anti-nef and without irreducible curves of arithmetic genus less or equal to $\ell>0$. Indeed the adjunction formula gives
$D^2=2\pi_D-2-D.K_X\geq 2\pi_D-2\geq 2\ell$.
\end{itemize}

\begin{lemma}\label{lemma_quadra}
Let $X$ be a surface on which any ${\mathbb F}_q$-irreducible curve has self-intersection at least $\beta\geq 0$. 
Assume that $H$ is a rational effective ample divisor on $X$ and let $r$ be a positive integer. 
Let $D=\sum^k_{i=1} n_iD_i$
be the decomposition as a sum of ${\mathbb F}_q$-irreducible and reduced curves of an effective divisor $D$ linearly equivalent to $rH$. Then we have
\begin{enumerate}
\item\label{k_beta} if $\beta >0$ then $k \leq r\sqrt{\frac{H^2}{\beta}}$;
\item $\sum_{i=1}^k (2\pi_{D_i}-2-D_i.K_X)\leq \varphi(k)$, with
\begin{equation}\label{varphi(k)}
\varphi(k)\coloneqq (k-1)\beta+\left(r\sqrt{H^2}-(k-1)\sqrt{\beta}\right)^2.
\end{equation}
\end{enumerate}
\end{lemma}

\begin{proof}
Since by hypothesis we have 
$\sqrt{\beta} \leq \sqrt{D_i^2}$, we deduce that 
$k\sqrt{\beta} \leq   \sum_{i=1}^k n_i\sqrt{D_i^2}$.
By (\ref{Hodge}), we get $k\sqrt{\beta} \leq   \sum_{i=1}^k n_i\frac{D_i.H}{\sqrt{H^2}}= \frac{rH.H}{\sqrt{H^2}}=r\sqrt{H^2}$, 
from which the first item follows.

Setting $x_i\coloneqq\sqrt{2\pi_{D_i}-2-D_i.K_X}$, we have
by adjunction formula $x_i = \sqrt{D_i^2}\geq \sqrt{\beta}$.
Moreover the previous inequalities ensure that
$\sum_{i=1}^k x_i\leq \sum_{i=1}^k n_i\sqrt{D_i^2}\leq r\sqrt{H^2}$.
Then, the maximum of $\sum_{i=1}^k (2\pi_{D_i}-2-D_i.K_X)=\sum_{i=1}^k x_i^2$ under the conditions $x_i\geq\sqrt{\beta}$ and  $\sum_{i=1}^kx_i \leq r\sqrt{H^2}$ is reached when each but one $x_i$ equals the minimum $\sqrt{\beta}$, and only one is equal to $r\sqrt{H^2}-(k-1)\sqrt{\beta}$, and this concludes the proof.
\end{proof}

\begin{theorem}\label{theorem_beta}
Let $X$ be a surface on which any ${\mathbb F}_q$-irreducible curve has self-intersection at least $\beta>0$. 
Consider a rational effective  ample divisor $H$ on $X$ avoiding a set $S$ of rational points and let $r$ be a positive integer.
Then
$$d(X, rH, S) \geq 
\begin{cases}
\sharp S-\max\left\{\psi(1), \psi\left( r\sqrt{\frac{H^2}{\beta}}\right)\right\}
- \frac{m}{2} r\sqrt{\frac{H^2}{2\beta}} \text{ if $K_X$ is nef},\\
\sharp S-\max\left\{\psi(1), \psi\left( r\sqrt{\frac{H^2}{\beta}}\right)\right\} \text{ if $-K_X$ is nef}
\end{cases}
$$
with
$$\psi(k)\coloneqq\frac{m}{2}\varphi(k)+k(q+1+m),$$
where $\varphi(k)$ is given by equation~(\ref{varphi(k)}).
\end{theorem}

\begin{proof}
For any non-zero $f \in L(rH)$, we have by (\ref{N(f)<=}) and by point $(\ref{A-P})$  of Theorem \ref{irreduciblecurves} that
 $N(f) \leq k(q+1)+m\sum_{i=1}^k \pi_{D_i}.$
 Lemma~\ref{lemma_quadra} implies  that
$N(f) \leq k(q+1)  +\frac{m}{2}\bigl(2k+\varphi(k)+\sum_{i=1}^k D_i.K_X\bigr).$
In case $K_X$ is nef, we have $\sum_{i=1}^kD_i.K_X \leq \sum_{i=1}^kn_i D_i.K_X = rH.K_X$, and in case $-K_X$ is nef, we get
$\sum_{i=1}^kD_i.K_X \leq 0$, and the theorem follows.
\end{proof}

\subsection{Fibered surfaces with nef canonical divisor} \label{Section_fibration}
We consider in this subsection AG codes from fibered surfaces whose canonical divisor is nef.
We adopt the vocabulary of \cite[III, \S 8]{Silverman}
and we refer the reader to this text
for the basic notions we recall here.
A fibered surface is a 
surjective morphism $\pi: X\rightarrow B$
 from a smooth projective surface $X$ to a smooth absolutely irreducible curve $B$.
 We denote by $\pi_0$ the common arithmetic genus of the fibers and by $g_B$  the genus of the base curve $B$. 
Elliptic surfaces are among the first non-trivial examples of fibered surfaces.
For such surfaces we have $\pi_0=1$ 
and the canonical divisor is always nef 
(see \cite{Bombieri_Mumford}).

We recall that on a fibered surface every divisor 
can be  uniquely written
as a sum of \emph{horizontal} curves (that is mapped onto $B$ by $\pi$) and \emph{fibral} curves (that is  mapped onto a point by $\pi$).

\begin{lemma}\label{r_bound}
Let $\pi : X \rightarrow B$ be a fibered surface.
Let $H$ be a rational effective ample divisor on $X$ and let $r$ be a positive integer. For any effective divisor $D$ linearly equivalent to $rH$,
consider its decomposition $D=\sum^k_{i=1} n_iD_i$
as a sum of reduced ${\mathbb F}_q$-irreducible curves as in (\ref{decomposition}). 
Denote by $\overline{r}_i$  the number of absolutely irreducible components of $D_i$. Then, we have
$$\sum_{i=1}^{k}\overline{r}_i\leq rH^2.$$
\end{lemma}

\begin{proof}
Write $D=\sum_{i=1}^k n_i D_i=\sum_{i=1}^k n_i \sum_{j=1}^{\overline{r}_i}D_{i,j}$ where  the $D_{i,j}$ are the absolutely irreducible components of $D_i$. 

We use that $n_i>0$, that $D$ is numerically equivalent to $rH$ and that $D_{i,j}.H>0$ to get
$$\sum_{i=1}^{k}\overline{r}_i \leq \sum_{i=1}^k \sum_{j=1}^{\overline{r}_i}D_{i,j}.H\leq \sum_{i=1}^k n_i \sum_{j=1}^{\overline{r}_i}D_{i,j}.H= \sum_{i=1}^k n_i D_i.H=rH.H,$$
which proves the lemma.
\end{proof}

The next theorem involves the \emph{defect} $\delta(B)$ of a smooth absolutely irreducible curve $B$ defined over $\mathbb{F}_q$ of genus $g_B$, which is defined by
$$\delta(B) \coloneqq q+1+mg_B-\sharp B(\mathbb{F}_q).$$
 By the Serre-Weil theorem this defect is a non-negative number. The so-called maximal curves have defect $0$, and the smaller the number of rational points on $B$ is, the greater the defect is.

\begin{theorem}\label{theorem_fibration}
Let $\pi : X \rightarrow B$ be a fibered surface whose canonical divisor $K_X$ is nef.
Assume that $H$ is a rational effective ample divisor on $X$ having at least one horizontal component
and avoiding a set $S$ of rational points. For any positive integer $r$ 
the minimum distance of  $\mathcal{C}(X,rH, S)$ satisfies
$$d(X,rH, S)\geq d^*(X, rH, S) + \delta(B)$$
where $d^*(X, rH, S)$ is given by formula~(\ref{d*}).
\end{theorem}

Recall that the general bound we obtain in Theorem~\ref{ourbound} in Section \ref{Codes_over_algebraic_surfaces} for surfaces with nef canonical divisor is $d(X,rH, S)\geq  d^*(X, rH, S)$, thus the bound from Theorem \ref{theorem_fibration} is always equal or better. 
Actually Theorem \ref{theorem_fibration} is surprising, since it says that the lower bound for the minimum distance is all the more large because the defect $\delta(B)$ is. 
Consequently it looks like considering fibered surfaces on curves with few rational points and large genus could lead to potentially good codes.

\begin{proof}
Recall that for any non-zero $f\in L(rH)$, we have 
$d(X,rH, S) \geq \sharp S- N(f)$, and that
$N(f)\leq \sum_{i=1}^k \sharp D_i(\mathbb{F}_q)$ if we use the notation $D_f:=rH+(f)=\sum_{i=1}^k n_i D_i$ introduced in (\ref{decomposition}). We again denote by  $\overline{r}_i$  the number of absolutely irreducible components of $D_i$.
In order to introduce the ${\mathbb F}_q$-irreducible components of $D_f$, 
write $k=h+v$, where $h$ (respectively $v$) 
is the number of horizontal curves denoted by $H_1,\ldots,H_h$,
(respectively fibral curves denoted by $F_{1},\ldots,F_v$). 
Then we get
$N(f)\leq \sum_{i=1}^h \sharp H_i(\mathbb{F}_q)+\sum_{i=1}^v \sharp F_i(\mathbb{F}_q).$
Since $B$ is a smooth curve, the morphisms $H_i\rightarrow B$ are flat.
Now applying Proposition~\ref{covering} to horizontal curves and Theorem~\ref{irreduciblecurves} to fibral curves gives
\begin{align}\label{Nf_hv_0}
\begin{split}
N(f)&\leq h(\sharp B(\mathbb{F}_q)-mg_B)+m\sum_{i=1}^h \pi_{H_i}+q\sum_{i=1}^h(\overline{r}_i-1)+qv+v+m\sum_{i=1}^v \pi_{F_i}\\
&=h(\sharp B(\mathbb{F}_q)-mg_B-q)+m\sum_{i=1}^k \pi_{D_i}+q\sum_{i=1}^k \overline{r}_i+v,
\end{split}
\end{align}
where we used the fact that $v\leq \sum_{i=h+1}^k \overline{r}_i$. 

Since the canonical divisor of the fibered surface is assumed to be nef, Lemma~\ref{k_HodgeTGV_bounds} gives a bound for  $\sum_{i=1}^k \pi_{D_i}$. 
We set $v=k-h$ 
and we use Lemma~\ref{r_bound} with (\ref{Nf_hv_0}) to obtain
\begin{align*}\label{Nf_hv}
\begin{split}
N(f)&\leq h(\sharp B(\mathbb{F}_q)-mg_B-q)+m(\pi_{rH}-1)+mk+qrH^2+v\\
&=h(\sharp B(\mathbb{F}_q)-mg_B-q-1)+m(\pi_{rH}-1)+mk+qrH^2+k\\
&=-h\delta(B)+m(\pi_{rH}-1)+mk+qrH^2+k.
\end{split}
\end{align*}

Now, $D_f.F\equiv rH.F >0$ since 
$F$ is a generic fiber and
$rH$ is assumed to have at least one horizontal component.
Thus, $D_f$ has also at least one horizontal component, 
that is $h \geq 1$.
Moreover, again from Lemma~\ref{k_HodgeTGV_bounds} we have $k \leq rH^2$.
As the defect $\delta(B)$ is non-negative it follows that
$$N(f) \leq -\delta(B)+rH^2(q+1+m)+m(\pi_{rH}-1)$$
and the theorem is proved.
\end{proof}

\subsection{Fibered surfaces whose singular fibers are irreducible.}

In this subsection we drop off the condition on the canonical divisor.
Instead, we assume that every singular fiber on $X$ is ${\mathbb F}_q$-irreducible. 
To construct examples of such surfaces,
fix any $d \geq 1$ and recall that the dimension 
of the space of degree $d$ homogeneous polynomials  in three variables is
$\binom{d+2}{2}$. Hence the space ${\mathcal P}_d$ of plane curves of
degree $d$ is ${\mathcal P}_d={\mathbb P}^{\binom{d+2}{2}-1}$. 
Any curve $B$ drawn in ${\mathcal P}_d$ gives rise to a fibered surface,
whose fibers are plane curves of degree $d$, that is with $\pi_0=\frac{(d-1)(d-2)}{2}$. The locus of singular
curves being a subvariety of ${\mathcal P}_d$, choosing $B$ not contained in this singular locus yields to
a fibered surface with smooth
generic fiber. As the locus of reducible curves
has high codimension in ${\mathcal P}_d$, choosing $B$ avoiding this locus yields to fibered surfaces
without reducible fibers.

We consider the case where $\pi_0$ and $g_B$ are both at least $2$ and we set $\ell=\min(\pi_0,g_B)-1\geq 1$. 
We recall again that every divisor on $X$ can be uniquely 
written as a sum of horizontal and fibral curves. If we denote by $H$ an horizontal curve and by $V$ a fibral curve defined over ${\mathbb F}_q$, we have that $\pi_{H}\geq g_B$ 
and $\pi_{V}=\pi_0$.
Therefore, in this setting, $X$ contains no $\mathbb{F}_q$-irreducible curves defined over $\mathbb{F}_q$ of arithmetic genus smaller than or equal to $\ell$. Thus Lemma~\ref{lemma_ell} applies and gives the same bound for $\sum_{i=1}^k \pi_i$ as when $K_X$ is nef and the bound $k\leq (\pi_{rH}-1)/\ell$ for the number $k$ of ${\mathbb F}_q$-irreducible components of $D_f$. We consider this new bound for $k$ 
in the proof of Theorem~\ref{theorem_fibration}
and we get instead the following result.

\begin{theorem}\label{theorem_fibration_two}
Let $\pi : X \rightarrow B$ be a fibered surface.
We consider a rational effective ample divisor $H$ on $X$ having at least one horizontal component
and avoiding a set $S$ of rational points. Let $r$ be a positive integer.
We denote by $g_B$ the genus of $B$ and
by $\pi_0$ the arithmetic genus of the fibers and we set  $\ell=\min(\pi_0,g_B)-1$.
Suppose that every singular fiber is ${\mathbb F}_q$-irreducible
 and that $\ell \geq 1$. Then the minimum distance of $\mathcal{C}(X,rH, S)$ satisfies
$$d(X,rH, S)\geq d^*(X, rH, S) + \left(rH^2-\frac{\pi_{rH}-1}{\ell}\right)(q+1+m)+ \delta(B),$$
where $d^*(X, rH, S)$ is given by formula~(\ref{d*}).
\end{theorem}

Naturally this bound is better than the one in Theorem~\ref{theorem_fibration} if and only if $\pi_{rH}-1<\ell rH^2$.
Furthermore it  improves  the bound of Theorem \ref{theorem_ell} by the addition of the non-negative term $\delta(B)$.

\section{An example: surfaces in $\mathbb{P}^3$}\label{Hypersurfaces}

This section is devoted to the study of the minimum distance of AG codes
 over a surface $X$ of degree $d\geq 3$ embedded in $\mathbb{P}^3$.

We consider  the class $L$ of an hyperplane section of $X$. So $L$ is ample, $L^2=d$ and the canonical divisor on $X$ is $K_X=(d-4)L$ (see \cite[p.212]{Shafarevich}). 
In this setting, we fix a rational effective ample divisor $H$ and $r$ a positive integer.
We apply our former theorems to this context to
give bounds on the minimum distance of 
the code $\mathcal{C}(X,rH, S)$.

We recall that cubic surfaces 
are considered by Voloch and Zarzar in \cite{Voloch_Zarzar} and \cite{Zarzar}
to provide computationally good codes. 
In Section 4 of \cite{Little_Schenck}
Little and Schenck propose theoretical and experimental results for surfaces in $\mathbb{P}^3$
always in the prospect of finding good codes.
We also contribute to this study
with a view to bounding the minimum distance 
according to the geometry of the surface.

\begin{proposition}\label{without_hypothesis}
Let $X$ be a surface of degree $d\geq 3$ embedded in $\mathbb{P}^3$. 
Consider a rational effective ample divisor $H$ avoiding a set $S$ of rational points on $X$ and let $r$ be a positive integer.
Then the minimum distance of the code $\mathcal{C}(X,rH, S)$ satisfies
\begin{enumerate}
\item if $X$ is a cubic surface, then
$$d(X,rH, S)\geq d^*(X, rH, S)+mr(\pi_{H}-1).$$
\item If $X$ has degree $d\geq4$ then
$$d(X,rH, S)\geq d^*(X, rH, S),$$
\end{enumerate}
where
$$d^*(X, rH, S)= \sharp S-rH^2(q+1+m)-m(\pi_{rH}-1)$$
is the function defined in (\ref{d*}).
\end{proposition}

\begin{proof}
Since $K_X=(d-4)L$ we have for cubic surfaces that $K_X=-L$ and thus the canonical divisor is anti-ample, while for surfaces of degree $d\geq 4$ the canonical divisor is ample or the zero divisor, thus is nef. Hence we can apply Theorem \ref{ourbound} from which the proposition follows.
\end{proof}

\subsection{Surfaces in $\mathbb{P}^3$ without irreducible curves of low genus}

In the complex setting,
the Noether-Lefschetz theorem asserts that a general surface $X$ of degree $d\geq 4$ in $\mathbb{P}^3$
is such that $\Pic(X)=\mathbb{Z}L$,
where $L$ is the class of an hyperplane section
(see \cite{Griffith_Harris}).
Here, general means outside
a countable union of proper subvarieties
of the projective space parametrizing the  surfaces of degree $d$ in $\mathbb{P}^3$.
Even if we do not know an analog of this statement in our context,
it suggests us the strong assumptions
 we take in this subsection, namely in Lemma \ref{genus_on_hypersurfaces} 
 and Proposition \ref{hypersurfaces_wo_curves_small_genus}. 

\begin{lemma}\label{genus_on_hypersurfaces}
Let $X$ be a surface of degree $d\geq4$ 
in $\mathbb{P}^3$ 
of arithmetic Picard number one.
Suppose that $\NS(X)$ is generated 
by the class of an hyperplane section $L$. 
Consider an $\mathbb{F}_q$-irreducible curve $D$ on $X$
 of arithmetic genus $\pi_D$. Then $$\pi_D \geq (d-1)(d-2)/2.$$
\end{lemma}

\begin{proof}
Let $a$ be the integer such that $D=aL$ in $\NS(X)$. 
Since $D$ is an ${\mathbb F}_q$-irreducible curve and $L$ is ample, we must have $a>0$.
Then, using the adjonction formula, we get
\begin{align*}
2\pi_D-2&=D^2+D.K=a^2L^2+aL.(d-4)L\\
&=a^2d+ad(d-4)\geq d+d(d-4),
\end{align*}
and thus $\pi_D\geq (d-1)(d-2)/2$.
\end{proof}

By the previous lemma it is straightforward that
in our context  $X$ does not contain any
$\mathbb{F}_q$-irreducible curves 
of arithmetic genus smaller
 than or equal to $\ell$ for $\ell= (d-1)(d-2)/2-1=d(d-3)/2$. 
 This allows us to apply Theorem \ref{theorem_ell}, and get the following proposition.

\begin{proposition}\label{hypersurfaces_wo_curves_small_genus}
Let $X$ be a degree $d\geq4$ surface in $\mathbb{P}^3$ 
of arithmetic Picard number one whose N\'eron-Severi group $\NS(X)$
is generated by the class of an hyperplane section $L$.
Assume that $S$ is a set of rational points avoiding $L$. For any  positive integer $r$
 the minimum distance of the code $\mathcal{C}(X,rL, S)$ satisfies
$$d(X, rL, S) \geq d^*(X, rL, S, L)+rd\left(1-\frac{r+d-4}{d(d-3)}\right)(q+1+m)$$ 
where
$$d^*(X, rL, S, L)= \sharp S-rd(q+1+m)-mrd(r+d-4)/2.$$
\end{proposition}

\subsection{Surfaces in $\mathbb{P}^3$ of arithmetic Picard number one}
In this subsection 
we suppose that the arithmetic Picard number of $X$ is one, 
but we do not take the assumption that the N\'eron-Severi group 
is generated by an hyperplane section. 
Also in this case we can apply Theorem \ref{theorem_picard_number_one} 
which brings us to the following proposition.

\begin{proposition}\label{hypersurfaces_picard_1}
Let $X$ be a surface of degree $d\geq4$ in 
$\mathbb{P}^3$.
Assume that $\NS(X)=\mathbb{Z} H$ for an ample divisor $H$.
Consider $L=hH$, the class of an hyperplane section of $X$, for $h$ a positive integer. Let $S$ be a set of rational points on $X$ avoiding $H$ and let $r$ be a positive integer.
Then the minimum distance of the code $\mathcal{C}(X,rH, S)$ satisfies
$$d(X,rH, S)\geq
\begin{cases}
\sharp S-(q+1+m)-rH^2(r+h(d-4))/2 \text{ if } r>2(q+1+m)/mH^2,\\
\sharp S-r(q+1+m)-rH^2(1+h(d-4))/2 \text{ otherwise. } 
\end{cases}
$$

\end{proposition}
\begin{proof}
Since we have $3H^2+H.K_X=3H^2+H.(d-4)L=3H^2+h(d-4)H^2=H^2(3+h(d-4))\geq 0$, 
we can apply point \ref{fab} of Theorem \ref{theorem_picard_number_one} 
from which the proposition follows.
\end{proof}

\bigskip
\noindent
{\bf Acknowledgments:} The authors would like to thank the anonymous referee for relevant observations.

\printbibliography

\Addresses

\end{document}